\documentclass[12pt]{amsart}
\usepackage{amssymb, amsmath, amsthm} 
\usepackage{url}      
\usepackage{mathtools} 
\usepackage{tikz-cd}
\usepackage[T1]{fontenc}
\usepackage[utf8]{inputenc}
\usepackage{xcolor}
\usepackage{mathrsfs}
\usepackage{stmaryrd}
\usepackage{eucal}
\usepackage[all]{xy}
\usepackage[margin=1in]{geometry} 
\usepackage{hyperref}
\usepackage[dvipsnames]{xcolor}
\usepackage[shortlabels]{enumitem}
\hypersetup{bookmarksdepth=2}
\hypersetup{colorlinks=true}
\hypersetup{linkcolor=blue}
\hypersetup{citecolor=blue}
\hypersetup{urlcolor=blue}
\usepackage[style=alphabetic]{biblatex} 
\allowdisplaybreaks

\theoremstyle{plain}
\newtheorem{theorem}{Theorem}[section] 
\newtheorem{proposition}[theorem]{Proposition} 
\newtheorem{lemma}[theorem]{Lemma}

\theoremstyle{definition}
\newtheorem{definition}[theorem]{Definition}

\theoremstyle{remark}

\newtheorem*{acknowledgement}{Acknowledgements}

\renewcommand{\t}{\tau}

\newcommand{\FF}{\mathbb{F}}

\newcommand{\QQ}{\mathbb{Q}}

\newcommand{\ZZ}{\mathbb{Z}}

\newcommand{\End}{\operatorname{End}}

\newcommand{\GL}{\operatorname{GL}}

\newcommand{\grand}{{\textup{grand}}} 
\newcommand{\grano}{\mathcal{O}^{\grand}}

\newcommand{\id}{\operatorname{id}}

\newcommand{\Qbar}{{\bar{\QQ}}}

\renewcommand{\setminus}{\smallsetminus}

\title{Orbit Transversality for Abelian Varieties}

\author{Kaiwen Lu}
\address{Department of Mathematics, Brown University, Providence, RI, USA}
\email{kaiwen\_lu@brown.edu}

\begin{document}

\begin{abstract}
Let $X/K$ be an abelian variety defined over a number field and let $f:X\to X$ be a dominant morphism defined over $K$.
We show that $(f,K)$ is strongly dense orbit transversal. 
That is, up to replacing $K$ with a finite extension, every set of representatives for grand $(f,K)$-orbits is Zariski dense in $X$.

\end{abstract}

\maketitle

\section{Introduction}

Let $X$ be a variety over a field $K$ and let $f:X\to X$ be a morphism defined over $K$.
A central object of study in dynamics is the set of grand $(f,K)$-orbits of points in $X(K)$, as they encode the information of how points move under $f$.
\begin{definition}
    The \emph{grand $(f,K)$-orbit} of a point $P\in X(K)$ is the set of points whose orbits eventually merge with the orbit of $P$:
    \[\grano_{f,K}(P):=\{Q\in X(K)\mid  f^n(Q)=f^m(P) \text{ for some $n,m \ge 0$}\}.\]
\end{definition}
A natural question to ask is then: how are the grand $(f,K)$-orbits distributed in $X$?
We know that the set of grand $(f,K)$-orbits in $X(K)$ partitions $X(K)$, so we need to pin down a notion of distribution before we proceed.
To precisely discuss this problem, Pasten and Silverman \cite{pastensilverman} introduced orbit propagation principles, which measure the distribution of grand orbits by the Zariski density of some sets of representatives.
Using terminologies introduced by Silverman, one set of representatives considered is the following:
\begin{definition}
An \emph{$(f,K)$-transversal of $X$} is a subset $S\subseteq X(K)$ such that every $(f,K)$-grand orbit of $X(K)$ contains exactly one point in $S$.
\end{definition}
With the above definition, orbit propagation principles (C2$\exists$) and (C2$\forall$) described in \cite{pastensilverman} can be rephrased as follows:
\begin{definition}[Weak DOT]
We say that $(f,K)$ is \emph{weakly dense orbit transversal} (W-DOT) if for some finite extension $L/K$, \textbf{there exists} a Zariski dense $(f,L)$-transversal of $X$.
\end{definition}

\begin{definition}[Strong DOT]
We say that $(f,K)$ is \emph{strongly dense orbit transversal} (S-DOT) if for some finite extension $L/K$, \textbf{every} $(f,L)$-transversal of $X$ is Zariski dense.
\end{definition}

Pasten and Silverman proved orbit propagation principles of various strength for many classes of self-maps on families of projective varieties.
In the case of abelian varieties, they proved strongly dense orbit transversality for simple abelian varieties:
\begin{theorem}[\cite{pastensilverman} Theorem 6.1, rephrased]
Let $X/\Qbar$ be a geometrically simple abelian variety, and let $f:X\to X$ be an endomorphism of $X$ (as an abstract variety) such that there is a point in $X(\Qbar)$ with $f$-orbit that is Zariski dense in $X$.
Then there is a number field~$K/\QQ$ such that~$X$ and~$f$ are defined over~$K$ and $(f,K)$ is strongly dense orbit transversal.
\end{theorem}
The main result of this paper is strongly dense orbit transversality for dominant endomorphisms of arbitrary abelian varieties.
\begin{theorem}\label{main}
Let $X/\Qbar$ be an abelian variety, and let $f:X\to X$ be a dominant morphism.
Then there is a number field~$K/\QQ$ such that~$X$ and~$f$ are defined over~$K$, and $(f,K)$ is strongly dense orbit transversal.
\end{theorem}

Theorem \ref{main} strengthens and simplifies the proof of W-DOT for abelian varieties in the author's earlier unpublished paper \cite{wdot}.

As another motivation for the defintion of $(f,K)$-transversals, Pasten and Silverman proved that over a number field $K$ and up to replacing $K$ with a finite extension, if $X(K)$ is Zariski dense in $X$, then the existence of a dense $(f,K)$-transversal is equivalent to $X(K)\setminus(\Gamma_1\cup\cdots \cup \Gamma_r)$ being Zariski dense for every finite collection of grand $(f,K)$-orbits $\Gamma_1,\ldots,\Gamma_r$.
This gives more evidence for why the Zariski density of an $(f,K)$-transversal is a reasonable notion for measuring the distribution of grand $(f,K)$-orbits.

One thing to note is that their proof for the aforementioned equivalence crucially relies on the fact that $K$ is countable. 
In fact, if $K$ is uncountable, then we will prove that strongly dense orbit transversality always holds for dominant morphisms of irreducible varieties; see Proposition \ref{uncountablefield}.

Henceforth we assume that $K$ is a number field, and we may abuse notation when replacing $K$ with a finite extension and still call the extension $K$.

\begin{acknowledgement}
    I would like to thank my advisor Joseph Silverman for many helpful discussions and comments.
\end{acknowledgement}

\section{Proof of the main theorem}
Trying to work with an arbitrary transversal is difficult, so we turn the proof of the main theorem into an existence proof using the following lemma.

\begin{lemma}\label{equiv}
    Let $K$ be a field, $X/K$ a variety, and $f:X\to X$ a morphism.
    Then $(f,K)$ is strongly dense orbit transversal if and only if there exists a finite extension $L/K$ such that for every proper subvariety $V\subsetneq X$, there exists a point $P\in X(L)$ such that
    \[V\cap\grano_{f,L}(P)=\varnothing.\]
\end{lemma}
\begin{proof}
    Assume that $(f,K)$ is strongly dense orbit transversal.
    Let $V\subsetneq X$ be a proper subvariety.
    Let $L/K$ be a finite extension such that every $(f,L)$-transversal is Zariski dense and $V(L)\neq \varnothing$.
    If for every $P\in X(L)$, we have 
    \[V\cap\grano_{f,L}(P)\neq\varnothing,\]
    then we can construct an $(f,L)$-transversal $S\subseteq V(L)$, which has Zariski closure contained in $V$, hence is not Zariski dense in $X$, contradicting $(f,K)$ being strongly dense orbit transversal.
    
    Now assume that there exists a finite extension $L/K$ such that for all proper subvariety $V\subsetneq X$, there exists a point $P\in X(L)$ such that
    \[V\cap\grano_{f,L}(P)=\varnothing.\]
    Let $S\subseteq X(L)$ be an $(f,L)$-transversal.
    Suppose for the sake of contradiction that its Zariski closure $\overline{S}\subsetneq X$.
    By assumption, there exists a point $P\in X(L)$ such that
    \[\overline{S}\cap\grano_{f,L}(P)=\varnothing.\]
    On the other hand, the set $S$ being an $(f,L)$-transversal implies that there exists a point $s\in S\subseteq \overline{S}(L)$ such that $s\in \grano_{f,L}(P)$, which is a contradiction.
\end{proof}

We next prove two lemmas concerning the $\ZZ$-module structure of $X(K)$ and use $\ZZ$-linear algebra to obtain the orbit avoidance result needed to apply Lemma \ref{equiv}.
Note that Lemma \ref{orbitavoid} is purely a statement in $\ZZ$-linear algebra, and it will be applied to $X(K)$, which is a finitely generated $\ZZ$-module.

\begin{lemma}\label{corank}
    Let $X$ be an abelian variety over a number field $K$.
    For any finite number $\rho$,
    there exists a finite extension $L/K$ such that for all proper abelian subvarieties $Y\subsetneq X$, we have that $Y(L)$ has corank at least~$\rho$ in $X(L)$.
    In other words, the $\ZZ$-rank of the quotient $X(L)/Y(L)$ is at least $\rho$.
\end{lemma}
\begin{proof}
    By Poincar\'e reducibility theorem, the abelian variety $X$ is isogenous to a product $\prod_{i=1}^kB_i^{n_i}$, where $B_i$ is a simple abelian variety for all $i$.
    Take $L/K$ so that $\End_L(X)=\End_{\Qbar}(X)$ and the rank of $B_i(L)$ is at least $\rho$ for all $i$.
    This is possible because $\End_\Qbar(X)$ is a finitely generated $\ZZ$-module, and there exists a finite extension of $K$ where the generators are all defined.
    For each $B_i$ in the Poincar\'e decomposition, we can take an extension $L_i/K$ where the rank of $B_i(L)$ is at least $\rho$, then we take $L$ to be the compositum of all aforementioned extensions of $K$.

    Let $Y\subsetneq X$ be a proper abelian subvariety of $X$, then $Y$ is the image of an endomorphism in $\End_L(X)$.
    The quotient abelian variety $X/Y$ is isogenous to $\prod_{i=1}^kB_i^{m_i}$ for some $m_i$ not all 0, with an isogeny defined over $L$.
    Since $B_i(L)$ has rank at least $\rho$ for all $i$ and isogenies over $L$ preserve rank, $X(L)/Y(L)$ has rank at least $\rho$, and $Y(L)$ has corank at least $\rho$ in $X(L)$.
\end{proof}

\begin{lemma}\label{orbitavoid}
    Let $X$ be a free $\ZZ$-module of rank $R$ and $f:X\to X$ be an affine map such that $f(x)=T(x)+Q$, where $T$ is a linear map with $\det(T)\neq 0$.
    Let $V=\cup_{i=1}^N (a_i+Y_i)$ be a union of translates of submodules of $X$ such that the corank of $Y_i$ is at least~$2$ in $X$ for all~$i$.
    Then there exists a point $P\in X\setminus V$ such that $\grano_f(P)\cap V=\varnothing$.
\end{lemma}
\begin{proof}
    Consider the homogenization $F$ of $f$ on $X\oplus \ZZ e$ given by 
    \[F(x+te)=T(x)+tQ+te.\]
    Note that for any $v=x+e$ where $x\in X$, we have 
    \[F(x+e)=T(x)+Q+e=f(x)+e.\]
    Hence the restriction of $F$ on the submodule $X$ shifted by $e$ is $F|_{X+e}=f$.
    
    In matrix form, we find the matrix $H$ for the homogenization $F$ is $(R+1)\times(R+1)$ of the form
    \[H=\begin{bmatrix}
        T & Q\\
        0 & 1
    \end{bmatrix},\]
    which implies that $\det(H)=\det(T)\ne 0$.
    Since $\det(H)\neq 0$, we can take a prime $p\nmid \det(H)$, then 
    \[H\pmod{p}\in \GL_{R+1}(\FF_p)\] has finite order, say $h$, because $\GL_{R+1}(\FF_p)$ is finite.
    That is, there exists a matrix $M$ such that $H^h=\id+pM$.
    Take an integer $m>0$ such that $p^m>Nh$, we find that 
    \begin{equation}\label{pm}
        (H^h)^{p^m}=(\id+pM)^{p^m}=\sum_{i=0}^{p^m}{p^m\choose i}(pM)^{i}=\id +\sum_{i=1}^{p^m}{p^m\choose i}(pM)^{i}.
    \end{equation}
    By Lemma 2.3 of \cite{leethesis}, if $1\le i\le p^m$, we have
    \begin{equation}\label{coeff}
        \nu_p({p^m\choose i}p^i)=m-\nu_p(i)+i> m.
    \end{equation}
    Going forward, we will use $\overline{x}$ to denote $x\pmod{p^m}$ for any object $x$.
    Equations \eqref{pm} and \eqref{coeff} imply that reducing mod $p^m$, we have 
    \[(H^h)^{p^m}\equiv \id\pmod{p^m},\]
    and ${\overline{H}}$ has order at most $h\cdot p^m$, from which we deduce that ${\overline{f}}$ has order at most $h\cdot p^m$ as well.
    As a direct consequence, for any point $\overline{P}\in \overline{X}$, the grand orbit of $\overline{P}$ under $\overline{f}$ has size at most $h\cdot p^m$.

    Now we count the size of the grand orbits of points in $\overline{V}$:
    \begin{align*}
        \left|\grano_{\overline{f}}(\overline{V})\right|=&\left|\grano_{\overline{f}}\left(\bigcup_{i=1}^N (\overline{a_i+Y_i})\right)\right|\\
        \le& \sum_{i=1}^N\left|\grano_{\overline{f}}(\overline{a_i+Y_i})\right|\\
        \le& N\cdot (h\cdot p^m)(p^m)^{R-2}\\
        =&(Nh)(p^m)^{R-1}\\
        <&(p^m)^R\\
        =&\left|\overline{X}\right|.
    \end{align*}

    Therefore there exists a point $\overline{P}\in \overline{X}\setminus \overline{V}$ such that 
    \[\grano_{\overline{f}}(\overline{P})\cap \overline{V}=\varnothing,\] and hence a lift $P\in X\setminus V$ also satisfies $\grano_f(P)\cap V=\varnothing$.
\end{proof}
\begin{proof}[Proof of Theorem \ref{main}]
    Let $X/\Qbar$ be an abelian variety, and let $f:X\to X$ be a dominant morphism.
    By the rigidity theorem~\cite[Section~4, Corollary~1]{mumfordabvar}, we can decompose $f$ into a homomorphism followed by a translation:
    \[
        f(x)=\t(x)+y,
    \]
    where $\t$ is a homomorphism of abelian varieties and $y\in X(K)$ is a point.
    Consider the following commutative diagram of $\ZZ$-modules:
    \begin{equation}\label{free}
        \begin{tikzcd}
        X(K) \arrow[r, "f"] \arrow[d, "\pi"]      & X(K) \arrow[d, "\pi"]     \\
        X(K)/X(K)_{\mathrm{tors}} \arrow[r, "f'"] & X(K)/X(K)_{\mathrm{tors}}
        \end{tikzcd}
    \end{equation}
    where $f$ is of the form $f(x)=\t(x)+y$ with $\t$ a $\ZZ$-module homomorphism and $y\in X(K)$, $\pi$ is the quotient by $X(K)_{\mathrm{tors}}$ map, and $f'(x)=\overline{\t}(x)+\pi(y)$.
    Note that from the commutative diagram  \eqref{free}, we have 
    \[\pi(P)\notin \grano_{f',K}(\pi(Q))\implies P\notin \grano_{f,K}(Q).\]
    Therefore for any subvariety $V\subsetneq X$, to find a point with grand $(f,K)$-orbit avoiding $V(K)$, it suffices to find a point avoiding the image of $V(K)$ in $X(K)/X(K)_{\mathrm{tors}}$ and take a lift.
    
    We replace $K$ by a finite extension from Lemma \ref{corank} so that every proper abelian subvariety of $X$ has corank at least 2.
    For any proper subvariety $V\subsetneq X$, by Faltings's theorem \cite{faltings}, we find that 
    \[V\cap X(K)=\bigcup_{i=1}^N(a_i+Y_i(K)),\]
    where $a_i\in X(K)$ and $Y_i\subseteq X$ are abelian subvarieties.
    Since $V$ is a proper subvariety, the abelian subvariety $Y_i$ is proper for all $i$ as well.
    By our choice of $K$, the $\ZZ$-submodule $Y_i(K)$ has corank at least 2 for all $i$, and $f$ being dominant implies that $\t$ has nonzero determinant,
    so we may apply Lemma \ref{orbitavoid} to $X(K)$ and $V(K)$ to get that there always exists a point $P\in X(K)$ such that 
    \[V(K)\cap\grano_{f,K}(P)=\varnothing.\]
    Then we conclude the proof by observing that this is equivalent to strongly dense orbit transversality for $(f,K)$ by Lemma \ref{equiv}.
\end{proof}

\section{Varieties over uncountable fields}
Lastly, we prove strongly dense orbit transversality for dominant maps on irreducible varieties over uncountable fields.

\begin{proposition}\label{uncountablefield}
Let $K$ be an uncountable field, $X/K$ an irreducible variety such that $X(K)$ is Zariski dense in $X$, and let $f:X\to X$ be a dominant morphism.
Then every $(f,K)$-transversal for $X(K)$ is Zariski dense in $X$.
\end{proposition}
\begin{proof}
    Let $S\subseteq X(K)$ be an $(f,K)$-transversal for $X(K)$ and assume for the sake of contradiction that $Z:=\overline{S}\subsetneq X$, i.e., the Zariski closure of $S$ is a proper closed subset of $X$, and hence $\dim(Z)<\dim(X)$.
    Since morphisms cannot increase the dimension of a variety, we get that $\dim(f^n(Z))\le \dim (Z)<\dim(X)$ and $\overline{f^n(Z)}$ is a proper closed subset of $X$ for all $n\ge 0$.
    On the other hand, because $f$ is dominant, the preimage of any proper closed subset is still a proper closed subset.
    Therefore $f^{-m}(\overline{f^n(Z)})$ is a proper closed subset of $X$ for all $m,n\ge 0$.
    Now observe that by definition of $(f,K)$-transversals, we have 
    \[X(K)=\bigcup_{m\ge 0}\bigcup_{n\ge0} f^{-m}(f^n(S)).\]
    Taking the Zariski closure on both sides, we get
    \[X=\overline{X(K)}=\overline{\bigcup_{m\ge 0}\bigcup_{n\ge0} f^{-m}(f^n(S))}\subseteq\bigcup_{m\ge 0}\bigcup_{n\ge0} f^{-m}(\overline{f^n(Z)}),\]
    because the rightmost term is closed and contains $\bigcup_{m\ge 0}\bigcup_{n\ge0} f^{-m}(f^n(S))$.
    This implies that $X$ is a countable union of proper subvarieties of $X$, which is a contradiction, so $\overline{S}=X$.
\end{proof}


\printbibliography

@article {pastensilverman,
    AUTHOR = {Pasten, Hector and Silverman, Joseph H.},
     TITLE = {Propagation of {Z}ariski dense orbits},
   JOURNAL = {Rev. Mat. Iberoam.},
  FJOURNAL = {Revista Matem\'{a}tica Iberoamericana},
    VOLUME = {42},
      YEAR = {2026},
    NUMBER = {3},
     PAGES = {855--902},
}

@book {mumfordabvar,
    AUTHOR = {Mumford, David},
     TITLE = {Abelian varieties},
    SERIES = {Tata Institute of Fundamental Research Studies in Mathematics},
    VOLUME = {5},
      NOTE = {With appendices by C. P. Ramanujam and Yuri Manin,
              Corrected reprint of the second (1974) edition},
 PUBLISHER = {Tata Institute of Fundamental Research, Bombay; by Hindustan
              Book Agency, New Delhi},
      YEAR = {2008},
     PAGES = {xii+263},
      %ISBN = {978-81-85931-86-9; 81-85931-86-0},
   MRCLASS = {14Kxx},
  MRNUMBER = {2514037},
}

@incollection {faltings,
    AUTHOR = {Faltings, Gerd},
     TITLE = {The general case of {S}. {L}ang's conjecture},
 BOOKTITLE = {Barsotti {S}ymposium in {A}lgebraic {G}eometry ({A}bano
              {T}erme, 1991)},
    SERIES = {Perspect. Math.},
    VOLUME = {15},
     PAGES = {175--182},
 PUBLISHER = {Academic Press, San Diego, CA},
      YEAR = {1994},
}

@book {leethesis,
    AUTHOR = {Lee, Pui Hang},
     TITLE = {Local {F}ields, {I}terated {E}xtensions, and {J}ulia {S}ets},
      NOTE = {Thesis (Ph.D.)--University of Hawai'i at Manoa},
 PUBLISHER = {ProQuest LLC, Ann Arbor, MI},
      YEAR = {2025},
     PAGES = {45},
      %ISBN = {979-8293-87921-2},
   MRCLASS = {99-05},
  MRNUMBER = {4981560},
}

@misc{wdot,
      title={On Dense Orbit Transversality for Endomorphisms of Abelian Varieties}, 
      author={Kaiwen Lu},
      year={2026},
      eprint={2606.29057},
      archivePrefix={arXiv},
}

\end{document}